\documentclass{report}
\usepackage{amsmath,amssymb,titling,amsthm,enumitem,mathrsfs,euscript,upgreek,wasysym,tikz-cd,yfonts,hyperref,cleveref,graphicx,stmaryrd,bbold,yfonts}
\usepackage[utf8]{inputenc}
\usepackage[margin=1in]{geometry}
\usepackage[all,cmtip,2cell]{xy}
\UseAllTwocells
\xyoption{2cell}
\xyoption{rotate}

\usepackage{tikz}
\RequirePackage{tikz}

\pgfdeclarelayer{edgelayer}
\pgfdeclarelayer{nodelayer}
\pgfdeclarelayer{foreground}
\pgfdeclarelayer{background}
\pgfsetlayers{background,edgelayer,nodelayer,main,foreground,background}

\def\thickness{0.7pt}

\tikzstyle{dot}=[circle, draw=black, fill=black!25, inner sep=.4ex, line width=\thickness, node on layer=foreground]
\tikzstyle{blackdot}=[dot, fill=black!50]
\tikzstyle{blackdot}=[dot, fill=gray!40!white]
\tikzstyle{whitedot}=[dot, fill=white]
\tikzstyle{reddot}=[dot, fill=red]
\tikzstyle{greendot}=[dot, fill=green]

\makeatletter
\pgfkeys{%
  /tikz/on layer/.code={
    \pgfonlayer{#1}\begingroup
    \aftergroup\endpgfonlayer
    \aftergroup\endgroup
  },
  /tikz/node on layer/.code={
    \gdef\node@@on@layer{%
  \setbox\tikz@tempbox=\hbox\bgroup\pgfonlayer{#1}\unhbox\tikz@tempbox\endpgfonlayer\egroup}
\aftergroup\node@on@layer
  },
  /tikz/end node on layer/.code={
    \endpgfonlayer\endgroup\endgroup
  }
}
\def\node@on@layer{\aftergroup\node@@on@layer}
\makeatother

\newlength\morphismheight
\newlength\minimummorphismwidth
\newlength\stateheight
\newlength\minimumstatewidth
\newlength\connectheight
\tikzset{width/.initial=\minimummorphismwidth}

\newif\ifwedge\pgfkeys{/tikz/wedge/.is if=wedge}
\tikzset{wedge}
\newif\ifvflip\pgfkeys{/tikz/vflip/.is if=vflip}
\newif\ifhflip\pgfkeys{/tikz/hflip/.is if=hflip}
\newif\ifhvflip\pgfkeys{/tikz/hvflip/.is if=hvflip}
\newif\ifconnectnw\pgfkeys{/tikz/connect nw/.is if=connectnw}
\newif\ifconnectne\pgfkeys{/tikz/connect ne/.is if=connectne}
\newif\ifconnectsw\pgfkeys{/tikz/connect sw/.is if=connectsw}
\newif\ifconnectse\pgfkeys{/tikz/connect se/.is if=connectse}
\newif\ifconnectn\pgfkeys{/tikz/connect n/.is if=connectn}
\newif\ifconnects\pgfkeys{/tikz/connect s/.is if=connects}
\newif\ifconnectnwf\pgfkeys{/tikz/connect nw >/.is if=connectnwf}
\newif\ifconnectnef\pgfkeys{/tikz/connect ne >/.is if=connectnef}
\newif\ifconnectswf\pgfkeys{/tikz/connect sw >/.is if=connectswf}
\newif\ifconnectsef\pgfkeys{/tikz/connect se >/.is if=connectsef}
\newif\ifconnectnf\pgfkeys{/tikz/connect n >/.is if=connectnf}
\newif\ifconnectsf\pgfkeys{/tikz/connect s >/.is if=connectsf}
\newif\ifconnectnwr\pgfkeys{/tikz/connect nw </.is if=connectnwr}
\newif\ifconnectner\pgfkeys{/tikz/connect ne </.is if=connectner}
\newif\ifconnectswr\pgfkeys{/tikz/connect sw </.is if=connectswr}
\newif\ifconnectser\pgfkeys{/tikz/connect se </.is if=connectser}
\newif\ifconnectnr\pgfkeys{/tikz/connect n </.is if=connectnr}
\newif\ifconnectsr\pgfkeys{/tikz/connect s </.is if=connectsr}
\tikzset{keylengthnw/.initial=\connectheight}
\tikzset{keylengthn/.initial =\connectheight}
\tikzset{keylengthne/.initial=\connectheight}
\tikzset{keylengthsw/.initial=\connectheight}
\tikzset{keylengths/.initial =\connectheight}
\tikzset{keylengthse/.initial=\connectheight}
\tikzset{connect nw length/.style={connect nw=true, keylengthnw={#1}}}
\tikzset{connect n length/.style ={connect n =true, keylengthn ={#1}}}
\tikzset{connect ne length/.style={connect ne=true, keylengthne={#1}}}
\tikzset{connect sw length/.style={connect sw=true, keylengthsw={#1}}}
\tikzset{connect s length/.style ={connect s =true, keylengths ={#1}}}
\tikzset{connect se length/.style={connect se=true, keylengthse={#1}}}
\tikzset{connect nw < length/.style={connect nw <=true, keylengthnw={#1}}}
\tikzset{connect n < length/.style ={connect n <=true,  keylengthn ={#1}}}
\tikzset{connect ne < length/.style={connect ne <=true, keylengthne={#1}}}
\tikzset{connect sw < length/.style={connect sw <=true, keylengthnw={#1}}}
\tikzset{connect s < length/.style ={connect s <=true,  keylengths ={#1}}}
\tikzset{connect se < length/.style={connect se <=true, keylengthse={#1}}}
\tikzset{connect nw > length/.style={connect nw >=true, keylengthnw={#1}}}
\tikzset{connect n > length/.style ={connect n >=true,  keylengthn ={#1}}}
\tikzset{connect ne > length/.style={connect ne >=true, keylengthne={#1}}}
\tikzset{connect sw > length/.style={connect sw >=true, keylengthsw={#1}}}
\tikzset{connect s > length/.style ={connect s >=true,  keylengths ={#1}}}
\tikzset{connect se > length/.style={connect se >=true, keylengthse={#1}}}

\makeatletter
\pgfdeclareshape{morphism}
{
    \savedanchor\centerpoint
    {
        \pgf@x=0pt
        \pgf@y=0pt
    }
    \anchor{center}{\centerpoint}
    \anchorborder{\centerpoint}
    \saveddimen\savedlengthnw
    {
        \pgfkeysgetvalue{/tikz/keylengthnw}{\len}
        \pgf@x=\len
    }
    \saveddimen\savedlengthn
    {
        \pgfkeysgetvalue{/tikz/keylengthn}{\len}
        \pgf@x=\len
    }
    \saveddimen\savedlengthne
    {
        \pgfkeysgetvalue{/tikz/keylengthne}{\len}
        \pgf@x=\len
    }
    \saveddimen\savedlengthsw
    {
        \pgfkeysgetvalue{/tikz/keylengthsw}{\len}
        \pgf@x=\len
    }
    \saveddimen\savedlengths
    {
        \pgfkeysgetvalue{/tikz/keylengths}{\len}
        \pgf@x=\len
    }
    \saveddimen\savedlengthse
    {
        \pgfkeysgetvalue{/tikz/keylengthse}{\len}
        \pgf@x=\len
    }
    \saveddimen\overallwidth
    {
        \pgfkeysgetvalue{/tikz/width}{\minwidth}
        \pgf@x=\wd\pgfnodeparttextbox
        \ifdim\pgf@x<\minwidth
            \pgf@x=\minwidth
        \fi
    }
    \savedanchor{\upperrightcorner}
    {
        \pgf@y=.5\ht\pgfnodeparttextbox
        \advance\pgf@y by -.5\dp\pgfnodeparttextbox
        \pgf@x=.5\wd\pgfnodeparttextbox
    }
    \anchor{north}
    {
        \pgf@x=0pt
        \pgf@y=0.5\morphismheight
    }
    \anchor{north east}
    {
        \pgf@x=\overallwidth
        \multiply \pgf@x by 2
        \divide \pgf@x by 5
        \pgf@y=0.5\morphismheight
    }
    \anchor{east}
    {
        \pgf@x=\overallwidth
        \divide \pgf@x by 2
        \advance \pgf@x by 5pt
        \pgf@y=0pt
    }
    \anchor{west}
    {
        \pgf@x=-\overallwidth
        \divide \pgf@x by 2
        \advance \pgf@x by -5pt
        \pgf@y=0pt
    }
    \anchor{north west}
    {
        \pgf@x=-\overallwidth
        \multiply \pgf@x by 2
        \divide \pgf@x by 5
        \pgf@y=0.5\morphismheight
    }
    \anchor{connect nw}
    {
        \pgf@x=-\overallwidth
        \multiply \pgf@x by 2
        \divide \pgf@x by 5
        \pgf@y=0.5\morphismheight
        \advance\pgf@y by \savedlengthnw
    }
    \anchor{connect ne}
    {
        \pgf@x=\overallwidth
        \multiply \pgf@x by 2
        \divide \pgf@x by 5
        \pgf@y=0.5\morphismheight
        \advance\pgf@y by \savedlengthne
    }
    \anchor{connect sw}
    {
        \pgf@x=-\overallwidth
        \multiply \pgf@x by 2
        \divide \pgf@x by 5
        \pgf@y=-0.5\morphismheight
        \advance\pgf@y by -\savedlengthsw
    }
    \anchor{connect se}
    {
        \pgf@x=\overallwidth
        \multiply \pgf@x by 2
        \divide \pgf@x by 5
        \pgf@y=-0.5\morphismheight
        \advance\pgf@y by -\savedlengthse
    }
    \anchor{connect n}
    {
        \pgf@x=0pt
        \pgf@y=0.5\morphismheight
        \advance\pgf@y by \savedlengthn
    }
    \anchor{connect s}
    {
        \pgf@x=0pt
        \pgf@y=-0.5\morphismheight
        \advance\pgf@y by -\savedlengths
    }
    \anchor{south east}
    {
        \pgf@x=\overallwidth
        \multiply \pgf@x by 2
        \divide \pgf@x by 5
        \pgf@y=-0.5\morphismheight
    }
    \anchor{south west}
    {
        \pgf@x=-\overallwidth
        \multiply \pgf@x by 2
        \divide \pgf@x by 5
        \pgf@y=-0.5\morphismheight
    }
    \anchor{south}
    {
        \pgf@x=0pt
        \pgf@y=-0.5\morphismheight
    }
    \anchor{text}
    {
        \upperrightcorner
        \pgf@x=-\pgf@x
        \pgf@y=-\pgf@y
    }
    \backgroundpath
    {
        \pgfsetstrokecolor{black}
        \pgfsetlinewidth{\thickness}
        \begin{scope}
                \ifhflip
                    \pgftransformyscale{-1}
                \fi
                \ifvflip
                    \pgftransformxscale{-1}
                \fi
                \ifhvflip
                    \pgftransformxscale{-1}
                    \pgftransformyscale{-1}
                \fi
                \pgfpathmoveto{\pgfpoint
                    {-0.5*\overallwidth-5pt}
                    {0.5*\morphismheight}}
                \pgfpathlineto{\pgfpoint
                    {0.5*\overallwidth+5pt}
                    {0.5*\morphismheight}}
                \ifwedge
                    \pgfpathlineto{\pgfpoint
                        {0.5*\overallwidth + 15pt}
                        {-0.5*\morphismheight}}
                \else
                    \pgfpathlineto{\pgfpoint
                        {0.5*\overallwidth + 5pt}
                        {-0.5*\morphismheight}}
                \fi
                \pgfpathlineto{\pgfpoint
                    {-0.5*\overallwidth-5pt}
                    {-0.5*\morphismheight}}
                \pgfpathclose
                \pgfusepath{stroke}
        \end{scope}
        \ifconnectnw
            \pgfpathmoveto{\pgfpoint
                {-0.4*\overallwidth}
                {0.5*\morphismheight}}
            \pgfpathlineto{\pgfpoint
                {-0.4*\overallwidth}
                {0.5*\morphismheight+\savedlengthnw}}
            \pgfusepath{stroke}
        \fi
        \ifconnectne
            \pgfpathmoveto{\pgfpoint
                {0.4*\overallwidth}
                {0.5*\morphismheight}}
            \pgfpathlineto{\pgfpoint
                {0.4*\overallwidth}
                {0.5*\morphismheight+\savedlengthne}}
            \pgfusepath{stroke}
        \fi
        \ifconnectsw
            \pgfpathmoveto{\pgfpoint
                {-0.4*\overallwidth}
                {-0.5*\morphismheight}}
            \pgfpathlineto{\pgfpoint
                {-0.4*\overallwidth}
                {-0.5*\morphismheight-\savedlengthsw}}
            \pgfusepath{stroke}
        \fi
        \ifconnectse
            \pgfpathmoveto{\pgfpoint
                {0.4*\overallwidth}
                {-0.5*\morphismheight}}
            \pgfpathlineto{\pgfpoint
                {0.4*\overallwidth}
                {-0.5*\morphismheight-\savedlengthse}}
            \pgfusepath{stroke}
        \fi
        \ifconnectn
            \pgfpathmoveto{\pgfpoint
                {0pt}
                {0.5*\morphismheight}}
            \pgfpathlineto{\pgfpoint
                {0pt}
                {0.5*\morphismheight+\savedlengthn}}
            \pgfusepath{stroke}
        \fi
        \ifconnects
            \pgfpathmoveto{\pgfpoint
                {0pt}
                {-0.5*\morphismheight}}
            \pgfpathlineto{\pgfpoint
                {0pt}
                {-0.5*\morphismheight-\savedlengths}}
            \pgfusepath{stroke}
        \fi
        \ifconnectnwf
            \draw [forward arrow style] (-0.4*\overallwidth,0.5*\morphismheight)
                to (-0.4*\overallwidth,0.5*\morphismheight+\savedlengthnw);
        \fi
        \ifconnectnef
            \draw [forward arrow style] (0.4*\overallwidth,0.5*\morphismheight)
                to (0.4*\overallwidth,0.5*\morphismheight+\savedlengthne);
        \fi
        \ifconnectswf
            \draw [forward arrow style] (-0.4*\overallwidth,-0.5*\morphismheight-\savedlengthsw)
                to (-0.4*\overallwidth,-0.5*\morphismheight);
        \fi
        \ifconnectsef
            \draw [forward arrow style] (0.4*\overallwidth,-0.5*\morphismheight-\savedlengthse)
                to (0.4*\overallwidth,-0.5*\morphismheight);
    \fi
    \ifconnectnf
        \draw [forward arrow style] (0,0.5*\morphismheight)
            to (0,0.5*\morphismheight+\savedlengthn);
    \fi
    \ifconnectsf
        \draw [forward arrow style] (0,-0.5*\morphismheight-\savedlengths)
            to (0,-0.5*\morphismheight);
    \fi
    \ifconnectnwr
        \draw [reverse arrow style] (-0.4*\overallwidth,0.5*\morphismheight)
            to (-0.4*\overallwidth,0.5*\morphismheight+\savedlengthnw);
    \fi
    \ifconnectner
        \draw [reverse arrow style] (0.4*\overallwidth,0.5*\morphismheight)
            to (0.4*\overallwidth,0.5*\morphismheight+\savedlengthne);
    \fi
    \ifconnectswr
        \draw [reverse arrow style] (-0.4*\overallwidth,-0.5*\morphismheight-\savedlengthsw)
            to (-0.4*\overallwidth,-0.5*\morphismheight);
    \fi
    \ifconnectser
        \draw [reverse arrow style] (0.4*\overallwidth,-0.5*\morphismheight-\savedlengthse)
            to (0.4*\overallwidth,-0.5*\morphismheight);
    \fi
    \ifconnectnr
        \draw [reverse arrow style] (0,0.5*\morphismheight)
            to (0,0.5*\morphismheight+\savedlengthn);
    \fi
    \ifconnectsr
        \draw [reverse arrow style] (0,-0.5*\morphismheight-\savedlengths)
            to (0,-0.5*\morphismheight);
    \fi
}
}
\makeatother

\pagecolor{white}

\xymatrixrowsep{20mm}
\xymatrixcolsep{20mm}

\theoremstyle{plain}
\newtheorem{thm}{Theorem}[section]
\newtheorem{lem}[thm]{Lemma}

\makeatletter
\newcommand\varitem[1]{\bfem[\textbf{A\arabic{enumi}\rlap{$#1$}.}]%
  \edef\@currentlabel{A\arabic{enumi}{$#1$}}}
\makeatother

\crefname{lemma}{Lemma}{Lemmas}

\makeatletter
\newcommand{\icirc}{\mathbin{\mathpalette\make@small\oplus}}
\newcommand{\smallotimes}{\mathbin{\mathpalette\make@small\otimes}}

\newcommand{\make@small}[2]{%
  \vcenter{\hbox{%
    \scalebox{0.6}{$\m@th#1#2$}%
  }}%
}
\makeatother

\theoremstyle{definition}
\newtheorem{defn}{Definition}[section]

\theoremstyle{remark}

\DeclareFontFamily{U}{skulls}{}
\DeclareFontShape{U}{skulls}{m}{n}{ <-> skull }{}

\DeclareFontFamily{U}{min}{}
\DeclareFontShape{U}{min}{m}{n}{<-> udmj30}{}

\begin{document}

{\centering\scshape\Large\textsc{An Axiomatic Approach to The Multiverse of Sets} \par}
{\centering\scshape\textsc{Alec Rhea} \par}

\begin{abstract}
Recent work in set theory indicates that there are many different notions of 'set', each captured by a different collection of axioms, as proposed by J. Hamkins in [Ham11]. In this paper we strive to give one class theory that allows for a simultaneous consideration of all set theoretical universes and the relationships between them, eliminating the need for recourse 'outside the theory' when carrying out constructions like forcing etc. We also explore multiversal category theory, showing that we are finally free of questions about 'largeness' at each stage of the categorification process when working in this theory -- the category of categories we consider for a given universe contains all large categories in that universe without taking recourse to a larger universe. We leverage this newfound freedom to define a category ${\bf Force}$ whose objects are universes and whose arrows are forcing extensions, a $2$-category $\mathcal{V}\mathfrak{erse}$ whose objects are the categories of sets in each universe and whose component categories are given by functor categories between these categories of sets, and a tricategory $\mathbb{Cat}$ whose objects are  the $2$-categories of categories in each universe and whose component bicategories are given by pseudofunctor, pseudonatural transformation and modification bicategories between these $2$-categories of categories in each universe.  \\
\end{abstract}

\tableofcontents

\chapter{Theory of the Multiverse}

Here we lay out the primitive notions under consideration, the language we will express them in, and the axioms these primitive notions obey. \\ 

\section{Primitive Notions}

The language is the first order language of class theory. The primitives are {\bf classes}, denoted with capital letters from the end of the alphabet $X,Y,Z,\dots$ together will the primitive relation of {\bf class membership}, denoted $\in$. For two classes $X$ and $Y$ we either have that $X$ {\bf is a member of $Y$}, denoted $X\in Y$, or not, denoted $X\notin Y$. We have a specified class called {\bf the multiverse}, denoted by an individual constant $\mathcal{M}$, and members $V\in\mathcal{M}$ of the multiverse are called {\bf universes}. A {\bf set} is a class $X$ that is a member of some universe, that is such that there exists some $V$ with $X\in V\in\mathcal{M}$. If we say that a class $X$ is a {\bf $V$-set} or a {\bf $V$-class}, we mean that $X\in V\in\mathcal{M}$ or $X\subseteq V\in\mathcal{M}$, respectively. \\

\section{Axioms}

The axioms are as follows. \\

\begin{itemize}
\item {\bf A1 - Class Extensionality}. Two classes are equal iff they have the same elements. $$\forall X\forall Y\big(X=Y\iff\forall Z(Z\in X\iff Z\in Y)\big).$$
\item {\bf A2 - Class Separation}. For any class $Z$ and any predicate $\phi(\cdot,Y)$, where $Y$ stands for finitely many class variables, there exists a class $A$ whose members are exactly those members of $Z$ satisfying $\phi(\cdot, Y)$. $$\forall Z\forall\phi(\cdot,Y)\exists A\forall X\Big(X\in A\iff\big(X\in Z\wedge \phi(X,Y)\big)\Big).$$ We denote the class $A$ guaranteed by this axiom together with a class $Z$ and predicate $\phi(\cdot,Y)$ by $$\{X\in Z:\phi(X,Y)\}.$$
\item {\bf A3 - Class Pairing}. For any two classes $X,Y$ there exists a class $Z$ whose members are $X$ and $Y$. $$\forall X\forall Y\exists Z\forall A(A\in Z\iff A=X\vee A=Y).$$ We denote the class $Z$ guaranteed by this axiom and two classes $X,Y$ by $\{X,Y\}.$ \\
\end{itemize}

\begin{defn}[{\it Singleton, Ordered Pair}]
We define the {\bf singleton} $$\{X\}=\{X,X\},$$ and the {\bf ordered pair} $$(X,Y)=\{\{X\},\{X,Y\}\}.$$ \\
\end{defn}

\begin{itemize}
\item {\bf A4 - Class Cartesian Product}. For any two classes $X,Y$ there exists a class $Z$ whose members are ordered pairs from $X$ and $Y$. $$\forall X\forall Y\exists Z\forall A\Big(A\in Z\iff\exists X'\in X\exists Y'\in Y\big(A=(X',Y')\big)\Big).$$ We refer to the class $Z$ guaranteed by this axiom and two classes $X,Y$ as the {\bf Cartesian product} of $X$ and $Y$, denoted $$X\times Y.$$ For each class $X$ and natural number $n$ we define $X^n$ to be the $n$-fold Cartesian product of $X$ with itself. \\
\end{itemize}

\begin{defn}[{\it Containment, Subclass, Relation}]
We say that a class $X$ {\bf is contained in} a class $Y$, written $X\subseteq Y$, iff all members of $X$ are also members of $Y$. $$X\subseteq Y\iff\forall Z(Z\in X\implies Z\in Y).$$ In this situation we may also say that $X$ is a {\bf subclass} of $Y$. A {\bf relation} is a class that is a subclass of some Cartesian product. $$R\ \text{is a relation}\iff \exists X\exists Y(R\subseteq X\times Y).$$ For classes $A,B$ we define a {\bf relation from $A$ to $B$} to be a subclass of $A\times B$. We define a {\bf relation on $A$} to be a subclass of $A\times A$. \\
\end{defn}

\begin{itemize}
\item {\bf A5 - Universes are Models}. Each universe $X\in\mathcal{M}$ consists of an ordered pair $(V,\in_V)$ such that $\in_V$ is a relation on $V$, called {\bf membership in $V$}. $$\forall X\Big(X\in\mathcal{M}\implies\exists(V,\in_V)\big(X=(V,\in_V)\wedge\in_V\subseteq V\times V\big)\Big).$$ We will abuse notation and denote a universe $(V,\in_V)\in\mathcal{M}$ by its first coordinate $V$. We say that a universe $(V,\in_V)$ {\bf models} a predicate $\phi$, written $$V\models\phi,$$ iff relativizing all quantifiers in $\phi$ to $V$ and replacing all instances of $\in$ with $\in_V$ yields a true sentence. We will say that a universe $V$ is {\bf standard} iff membership in $V$ is actual membership; that is, $$V\ \text{is standard}\iff\forall X\in V\forall Y\in V(X\in_V Y\iff X\in Y).$$ We say that a universe $V$ is {\bf transitive} iff members of members of $V$ are also members of $V$. $$V\ \text{is transitive}\iff\forall X\forall Y(X\in Y\in V\implies X\in V).$$ We say that a universe $V$ is {\bf complete} iff it is transitive and subclasses of $V$-sets are also $V$-sets. $$V\ \text{is complete}\iff V\ \text{is transitive}\wedge\forall X\forall Y(Y\subseteq X\in V\implies Y\in V).$$
\item {\bf A6 - Internal Empty Set}. Every universe thinks it has an empty set. $$\forall V\in\mathcal{M}\big(V\models\exists z\forall x(x\notin z)\big).$$
\item {\bf A7 - Well-Behaved Universe Existence}. A standard transitive universe exists. $$\exists V\in\mathcal{M}(V\ \text{is standard and transitive}).$$ \\
\end{itemize}

\begin{defn}[{\it Elementary Submodel}]
Let $V$ and $V'$ be universes. We will say that $V$ is an {\bf elementary submodel} of $V'$, written $$V\preceq V',$$ iff the following three conditions hold:
\begin{enumerate}
\item $V\subseteq V'$.
\item $\in_V=\in_{V'}\upharpoonleft\upharpoonright V$.
\item For any collection $\{x_i\}_{i<n}\subseteq V$ and any $n$-ary predicate $\phi$, we have that $$V'\models\phi(x_0,\dots,x_{n-1})\iff V\models\phi(x_0,\dots,x_{n-1}).$$ \\
\end{enumerate}
\end{defn}

Refer to the theory given by the language and primitives above together with axioms {\bf A1$-$A7} as $T_\emptyset$. For a fixed universe $V\in\mathcal{M}$, we will say that a predicate $\phi$ is {\bf safe above $V$} iff no universes containing $V$ occur in $\phi$, including $V$, and also the multiverse does not occur in it. \\

\begin{itemize}
\item {\bf A8 - Truth Closure}. For any universe $V$ whose existence is not instantiated by this axiom and for any predicate $\phi$ safe above $V$, if it is consistent with $T_\emptyset$ that $V$ models $\phi$ then there exists a universe $V+\phi$ such that $V+\phi\models\phi$ and $V\preceq V+\phi$. For every universe $V'$ instantiated by $n$ applications of this axiom at universes $\{V+\sum_{i<m}\phi_i\}_{m<n}$ and statements $\{\phi_i\}_{i<n}$, so $V'=V+\sum_{i<n}\phi_i$, let $T_\emptyset^{V'}$ be the theory given by $T_\emptyset$ plus individual constants $\{V+\sum_{i<m}\phi_i\}_{m\leq n}$ and axioms stating that $V'=V+\sum_{i<n}\phi_i$, that $V+\sum_{i<m}\phi_i\in\mathcal{M}$ for all $m\leq n$, that $V+\sum_{i<m}\phi_i\models\bigwedge_{i<m}\phi_i$ for all $m\leq n$, and that $V+\sum_{i<m}\phi_i\preceq V+\sum_{i<m+1}\phi_{i}$ for all $m<n$. For any statement $\phi'$ such that $V'$ modeling $\phi'$ is consistent with $T_\emptyset^{V'}$, there exists a universe $V'+\phi'$ such that $V'+\phi'\models\phi'$ and $V'\preceq V'+\phi'$.  \\
\end{itemize}

Note that for any statement $\phi$ independent of $T_\emptyset$, we will have universes $V+\phi$ and $V+\neg\phi$. This axiom allows us to obtain universes with basically any axioms we could want by observing that almost all statements we could write down about the standard transitive universe asserted to exist in {\bf A7} are consistent with $T_\emptyset$, but to show this we need models of $T_\emptyset$ where $\mathcal{M}$ is barren and models of $T_\emptyset$ where $\mathcal{M}$ is robust. To this extent, we would like some standard set theories to work with when constructing models of $T_\emptyset$.  \\

\begin{itemize}
\item {\bf A9 - Standard Theories}. For any well-established set theory $T$ (for example all variants of $ZFC$ with/without large cardinals, NF, KP, etc.) let $T_\wedge$ denote the conjunction of all axioms of $T$; then there exists a standard complete universe $V_T$ such that $$V_T\models T_\wedge.$$ \\
\end{itemize}

For example, recall that $ZFC$ is axiomatized as follows:

\subsubsection{Axioms of ZFC}
\begin{itemize}
\item {\bf ZFC 1}. $\forall x\forall y\big(x=y\iff\forall z(z\in x\iff z\in y)\big)$.
\item {\bf ZFC 2}. $\forall x\exists y(y\in x\wedge y\cap x=\emptyset).$
\item {\bf ZFC 3}. $\forall\phi(\cdot,y)\forall z\exists a\forall x\big(x\in a\iff x\in z\wedge\phi(x,y)\big)$.
\item {\bf ZFC 4}. $\forall x\forall y\exists z\forall a(a\in z\iff a=x\vee a=y)$.
\item {\bf ZFC 5}. $\forall x\forall y\exists z\forall a(a\in z\iff a\in x\vee a\in y)$.
\item {\bf ZFC 6}. $\forall \phi(\cdot,\cdot)\forall a\Big(\forall x\in a\exists y\big(\phi(x,y)\big)\implies\exists b\forall c\big(c\in b\iff\exists x\in a\big(\phi(x,c)\big)\big)\Big)$.
\item {\bf ZFC 7}. $\exists x\big(\emptyset\in x\wedge\forall y(y\in x\implies\mathcal{S}y\in x)\big)$.
\item {\bf ZFC 8}. $\forall x\exists y\forall z(z\subseteq x\iff z\in y)$.
\item {\bf ZFC 9}. $\forall x\Big(\emptyset\notin x\implies\exists f\big(f\ \text{is a function}\wedge dmn(f)=X\wedge rng(f)\subseteq\bigcup X\wedge \forall y\in x(f(y)\in y)\big)\Big)$. \\
\end{itemize}

{\bf A9} then tells us that there exists a standard complete universe $V_{ZFC}\in\mathcal{M}$ such that $$V_{ZFC}\models\bigwedge_{0<i<10}{\bf ZFC\ i},$$ so we can carry out all the standard set-theoretical constructions we would like to in $V_{ZFC}$. Obviously, for any large cardinal axiom $\phi$ we will have that $V_{ZFC+\phi}$ and $V_{ZFC}+\phi$ (where the first universe is given by {\bf A9} directly and the second universe used {\bf A9} to yield $V_{ZFC}$ and truth closure together with the consistency of $\phi$ over $ZFC$) both have exactly the same conditions imposed on them. Also worth noting is that because $ZF$ is incomplete (as all reasonably strong set theories are) there is no model of $ZF$ which models {\it exactly} the $ZF$-provable sentences, since the existence of such a model for a theory $T$ is equivalent to $T$ being complete. Put another way, any model $V_{ZF}$ will necessarily have sentences $\phi$ such that $V_{ZF}\models\phi$ but $ZF\nvdash\phi$ and $ZF\nvdash\neg\phi$, so these models will always 'model more' than can be proved from their underlying theories -- note that the additional statements they model in some sense pertain to the 'undetermined parts of the universe' for whatever given set theory we're considering (statements guaranteed by Gödel's incompleteness theorems for any set theory capable of defining the natural numbers and Gödel coding), since any reasonably strong set theory is incapable of describing the entire universe of sets it corresponds to [Göd31]. \\

\subsection{Additional Axioms for Classes}

We were careful above to only add axioms about classes as necessary to formalize the 'aether around the multiverse' (i.e. classes that aren't sets or universes or the multiverse itself) enough to enable a general discussion of modeling and elementary submodels. Here, we add two additional axioms for classes that aren't necessary for presenting any universes or constructing anything inside any specific universe, but rather for building class structure 'on top' of poorly structured universes of sets and the multiverse itself. \\

\begin{itemize}
\item {\bf A10 - Class Union}. For any class $X$, there exists a class $Y$ whose members are precisely the members of members of $X$. $$\forall X\exists Y\forall Z\big(Z\in Y\iff\exists A\in X(Z\in A)\big).$$ We will generally write the class $Y$ guaranteed by this axiom together with a class $X$ as $$\bigcup X$$ and refer to it as the {\bf union of $X$}. \\
\end{itemize}

\begin{defn}[{\it Binary Union}]
For classes $X,Y$ we define the {\bf binary union of $X$ and $Y$}, denoted $$X\cup Y,$$ by $$X\cup Y=\bigcup\{X,Y\}.$$
\end{defn}

\begin{lem}
$\forall X\forall Y\forall Z(Z\in X\cup Y\iff Z\in X\vee Z\in Y).$
\end{lem}
\begin{proof}
$$Z\in X\cup Y\iff\exists A\in\{X,Y\}(Z\in A),$$ $$A\in\{X,Y\}\iff A=X\vee A=Y,$$ $$Z\in A\wedge (A=X\vee A=Y)\iff Z\in X\vee Z\in Y.$$ \\
\end{proof}

\begin{itemize}
\item {\bf A11 - Powerclass}. For any class $X$ there exists a class $Y$ whose members are precisely the subclasses of $X$. $$\forall X\exists Y\forall Z(Z\subseteq X\iff Z\in Y).$$ We will generally denote the class $Y$ guaranteed by this axiom and a class $X$ by $$\mathcal{P}(X)$$ and refer to it as the {\bf powerclass of $X$}. \\
\end{itemize}

These axioms together will enable us to discuss things like topologies on $\mathcal{M}$, and these together with the next axiom will allow for the definition of a bicategory $\mathcal{V}\mathfrak{erse}$ whose objects are the categories of sets in each universe and whose component categories are given by functor categories between categories of sets. \\

\subsection{Structure Axioms}

The final axioms we consider are motivated by a desire to obtain a 'nicely structured' multiverse -- we achieve this goal by imposing minimum requirements on what a universe must have internally (in addition to having an internal empty set by {\bf A6}). \\

\begin{itemize}
\item {\bf A12 - Structured Multiverse}. All universes model {\bf A1}$-${\bf A4} and {\bf A10}. $$\forall V\in\mathcal{M}(V\models{\bf A1\wedge A2\wedge A3\wedge A4\wedge A10}).$$ We denote by {\bf A12*} the axiom consisting of {\bf A12} augmented with the requirement that all universes model {\bf A11}. \\
\end{itemize}

In particular, this ensures we have a notion of internal functions so that we can define categories, an internal category of sets, and an internal $2$-category of categories in a given universe. {\bf A12*} will further ensure that the categories of sets in each universe admit power objects. \\

\subsection{Consistency Strength}

For any theory $T$, denote by ${\sf Con}(T)$ the assertion that $T$ is consistent; that is, $${\sf Con}(T)\iff\nexists\phi(T\vdash\phi\wedge T\vdash\neg\phi).$$ Recall that $${\sf Con}(T)\iff\exists X(X\models T),$$ where $X$ is a class equipped with functions and relations for all function and relation symbols in $T$. {\bf A9} immediately and trivially yields that the theory of the multiverse proves the consistency of all known set theories with all known large cardinal hypotheses, since we have models of all of them. This includes models of mutually exclusive set theories like $ZFC$ and $ZF+AD$, and pretty much any other set theory we could write down -- as a consequence of modeling all these theories and thereby proving their consistency, the theory of the multiverse is greater in consistency strength than any known set theory. We could increase consistency strength further by assuming the consistency of this theory or trying to find some 'class large cardinal' axioms (or something like this), but for now this misses the point of formalizing the multiverse -- we would like a universal 'metatheory' in which to carry out all set-theoretic arguments involving multiple set theories and multiple universes of sets/models of those theories, and the theory of the multiverse as presented above serves this metatheoretic purpose well.  \\

\chapter{Set Theory in the Multiverse}

In this section, we briefly outline one easy way to use the theory of the multiverse as a universal metatheory for all set-theoretic constructions involving multiple set theories and universes/models of those theories. \\

\section{Models of $T_\emptyset$}

Basically the only new set-theoretical axiom introduced in this formalism is {\bf A8}, truth closure, and this axiom makes reference to the theory $T_\emptyset$ given by the primitive notions and constant symbols in the theory of the multiverse together with axioms {\bf A1}$-${\bf A7}. Specifically, invoking truth closure begins with taking the standard transitive universe $V$ asserted to exist in {\bf A6} and observing that $V$ modeling almost any statement we can write down is independent of $T_\emptyset$, yielding a plethora of new universes that are elementary extensions of $V$ where all these statements hold. To do this, we would naturally like a very weak model $W$ of $T_\emptyset$ where $\mathcal{M}^W$ is a singleton containing only $V^W$ and $V^W$ itself is as barren as possible, and another rich model $R$ where $\mathcal{M}^R$ is highly populated and $V^R$ is very robust.  \\

\begin{lem}[{\it $T_\emptyset$ has a weak model}]
Let $V_\omega$ denote the set of hereditarily finite sets, define $$V^{V_\omega}=\{\emptyset,1,2,3,\{3\}\},$$ and define $\mathcal{M}^{V_\omega}=\{V^{V_\omega}\}.$ Then $$(V_\omega,\in,\mathcal{M}^{V_\omega})\models T_\emptyset$$ and $V^{V_\omega}$ doesn't model separation, pairing, Cartesian product, replacement, union or powerset. \\
\end{lem}

\begin{lem}[{\it $T_\emptyset$ has a rich model}]
Let $\lambda$ be a Mahlo cardinal with $\kappa'<\lambda$ uncountable and inaccessible, denote by $V_\lambda$ the cumulative hierarchy below $\lambda$, define $$V^{V_\lambda}=V_{\kappa'},$$ and define $\mathcal{M}^{V_\lambda}=\{V_\kappa:\kappa\ \text{is inaccessible}\wedge\kappa<\lambda\}$. Then $$(V_\lambda,\in,\mathcal{M}^{V_\lambda})\models T_\emptyset$$ and $V^{V_\lambda}$ models all of $ZF$. \\
\end{lem}

These arguments together show that the standard transitive universe $V$ asserted to exist in {\bf A6} modeling any of the axioms of $ZF$ (besides extensionality) is independent of $T_\emptyset$, thus we have a universe modeling all of these axioms by truth closure -- what theory did we define the above models in, though? We could assume we were working in $ZFC$ plus sufficient large cardinal axioms, but that would defeat the whole point of the multiverse as a metatheory for set theory -- we would once again be stepping 'outside the theory' to construct models of it, requiring metatheoretic considerations outside the multiverse (even if these considerations could then be reproduced inside the multiverse). This is why we added {\bf A9} -- we can now assert that we were simply working in $V_{ZFC+\phi}$ for a sufficiently strong large cardinal axiom $\phi$ when constructing the above models, keeping all considerations 'in house'. Another rich model of $T_\emptyset$ is given by the collection of countable computably saturated models of $ZFC$ as defined in [GH11]. \\

\section{Mathematics in the Multiverse}

As we will see in the next section, the theory of the multiverse allows for a robust discussion of what would otherwise be 'proper classes' in each universe (and 'super-classes', 'super-duper classes', etc.) by allowing them to exist as classes outside that universe. It often occurs that algebraic structures we would like to discuss fail to have an underlying set in a given set theory, instead possessing an underlying class (or superclass etc.). The theory of the multiverse allows us to treat such algebraic structures exactly as we would intuitively like to, leveraging the fact that they all exist as well-behaved classes 'above' the universes we wish they lived in. \\

\begin{defn}[{\it Good Classes}]
For a fixed universe $V$, we say that a class $X$ is {\bf $V$-good} iff there exists a natural number $n$ such that $X\in\mathcal{P}^n(V)$, {\bf $V$-good of rank $n_0$} iff $n_0$ is the smallest natural number such that $X\in\mathcal{P}^{n_0}(V)$, {\bf good of rank $n_0$} iff there exists a universe $W$ such that $X$ is $W$-good of rank $n_0$, and {\bf good} iff there exists a universe $W$ such that $X$ is $W$-good. We say that {\bf X} is {\bf pseudo-good} iff there exists a natural number $n$ such that $X\in\mathcal{P}^n(\bigcup\mathcal{M})$, and {\bf pseudo-good of rank $n_0$} iff $n_0$ is the smallest natural number such that $X\in\mathcal{P}^{n_0}(\bigcup\mathcal{M})$. We say that $X$ is {\bf esoteric} iff there exists a natural number $n$ such that $X\in\mathcal{P}^n(\mathcal{M})$, and {\bf esoteric of rank $n_0$} iff $n_0$ is the smallest natural number such that $X\in\mathcal{P}^{n_0}(\mathcal{M})$. We say that $X$ is {\bf strange} iff it is not good, pseudo-good, or esoteric. \\
\end{defn}

As discussed in [Mul01] and originally defined in [LV61], for sufficiently structured universes $V$ (for example $V_{ZFC}$) we automatically have that all good classes exist, and that the Cartesian product of good classes is good, etc., all {\it without} the axioms for classes guaranteeing their existence -- the axioms we have for classes are thusly superfluous for talking about $V$-good classes over 'sufficiently nice' universes $V$. We added these axioms anyways because we wanted to be able to manipulate all of $\mathcal{M}$ and powerclasses thereof, and because poorly structured universes may otherwise have poorly behaved good classes above them. All of 'standard large mathematics' takes place in good classes -- note that all sets are good classes of rank $0$, $V$-sets are $V$-good classes of rank $0$, each universe $V$ and all the $V$-ordinals $O_n^V$ and all the $V$-surreals $N_0^V$ are $V$-good classes of rank $1$, topologies on $V$ and $O_n^V$ and $N_0^V$ are $V$-good classes of rank $2$, uniformities on $V$ and $O_n^V$ and $N_0^V$ are $V$-good classes of rank $3$, so on and so forth. Many constructions we care about involving 'all the somethings' in a given universe $V$ exist as $V$-good classes of varying rank over that universe, including categories of $V$-large $V$-categories, $V$-functor categories between $V$-large $V$-categories, so on and so forth. Further, we have that all universes are esoteric classes of rank $0$, the multiverse is an esoteric class of rank $1$, topologies on the multiverse are esoteric classes of rank $2$, uniformities on the multiverse are esoteric classes of rank $3$, and the categories constructed in the next section are all tuples of good or pseudo-good classes of varying rank. Pseudo-good classes that aren't good don't play a large role in any area of mathematics outside this paper that the author is aware of, but we singled them out because they play a large role in chapter 3 and are 'easily understood' as unions of good classes -- note that the multiverse is pseudo-good of rank $2$ in addition to being esoteric of rank $1$ since $X\subseteq\mathcal{P}(\bigcup X)\implies X\in\mathcal{P}(\mathcal{P}(\bigcup X))$ for all classes $X$, so we will generally have that the pseudo-good rank of a class is one higher than the esoteric rank whenever esoteric rank exists. Wether or not any strange classes exist is unclear, but if so they certainly deserve the moniker 'strange' -- they never appear in or above any set theoretical universe irrespective of the axioms used to define that universe, and they never appear as collections of universes or collections of collections of universes, etc., and they never appear as unions of any of these classes or their elements -- one wonders where they would appear.  \\

\chapter{Multiversal Category Theory}

Category theory has been intimately tied to multiversal questions almost since its inception, initially in the guise of which categories are 'small' -- it became apparent during the era of MacLane/Grothendieck et. al. that category theory seemed to require multiple successive 'universes', each containing the preceding one and respecting its notion of 'truth', in order to talk about 'all the groups/categories/etc.' at each universal stage. The theory of the multiverse handles this at a primitive level, taking multiple universes for granted and allowing us to consider the categories of all sets/groups/categories etc. at each stage by building them 'on top' of the universe they live in, using classes. It would also be possible to use truth closure to expand the universe we're in and grab whatever additional sets we need to talk the way we'd like to, as is essentially done when using Grothendieck universes etc., however we won't need to 'cheat' in this fashion by virtue of the fact that all classes occurring in some stage of the iterated powerclass of a sufficiently 'nice' universe are themselves 'nice'; these provide the structure we need 'on top of a universe' to talk about 'all the blorps' in a given universe. In what follows, if we say 'category', '$2$-category', etc. without a '$V$-' prefix, we mean the definition using classes and external membership (not sets/membership in any specific universe). \\

\section{The Forcing Category}

Up first, we define a category whose objects are universes and whose arrows are forcing extensions between universes. \\

\begin{defn}[{\it Forcing Category}]
We define a category $${\bf Force}$$ called the {\bf forcing category} as follows:
\begin{enumerate}
\item The objects of {\bf Force} are universes. That is, $${\bf Ob}_{\bf Force}=\mathcal{M}.$$
\item The arrows $V\to W$ of {\bf Force} are forcing extensions from $V$ to $W$.
\item Identities are given by trivial forcings.
\item Composition is given by composition of forcing extensions. \\
\end{enumerate}
\end{defn}

Slice categories ${\bf Force}/V$ thusly give all universes that we can force from into $V$, and coslice categories $V/{\bf Force}$ give all universes that we can force into from $V$. There are many other interesting things to be said about {\bf Force}, but for now we move on to higher categorical considerations. \\

\section{The $2$-Category of Universes}

Next, we define a $2$-category whose objects are the categories of sets in each universe and whose component categories are given by functor categories between these categories of sets. Recall that a relation $R$ from $X$ to $Y$ is {\bf entire} iff all members of $X$ occur at least once as a first coordinate in $R$, {\bf functional} iff each member of $X$ occurs at most once as a first coordinate in $R$, and a {\bf function from $X$ to $Y$} is a relation from $X$ to $Y$ which is entire and functional. We denote a function $f:X\to Y$ in class builder notation as $$f=\langle f(X')\in Y:X'\in X\rangle.$$ Up first, we must define what we mean by the 'category of sets in a universe'. In what follows, we assume {\bf A12}.  \\

\begin{defn}[{\it Category of Sets in a Universe}]
Let $V$ be a universe. We define a category $${\bf Set}_V$$ called the {\bf category of $V$-sets} as follows:
\begin{enumerate}
\item The objects of ${\bf Set}_V$ are $V$-sets. That is, $${\bf Ob}_{{\bf Set}_V}=V.$$
\item The arrows $f:X\to Y$ in ${\bf Set}_V$ are functions from $X$ to $Y$ in $V$. That is, $${\bf Hom}_{{\bf Set}_V}(X,Y)=\{f\in V:V\models f\ \text{is a function from}\ X\ \text{to}\ Y\},$$ and we define $${\bf Hom}_{{\bf Set}_V}=\big\{{\bf Hom}_{{\bf Set}_V}(X,Y)\in\mathcal{P}(V):(X,Y)\in V\times V\big\}.$$
\item Identity arrows are given by identity functions. That is, for all $X\in V$ we define $$1_X=\langle x\in_V X:x\in_V X\rangle,$$ and we further define $${\bf 1}^V=\{1_X\in V:X\in V\}.$$
\item Composition is given by internal composition of functions. That is, for $V$-functions $g:X\to Y$, $f:Y\to Z$, we define $$f\circ_{_V}g=\langle f(g(x))\in_VZ:x\in_VX\rangle,$$ for each triplet of objects $X,Y,Z\in {\bf Ob}_{{\bf Set}_V}$ we define $$\circ_{XYZ}=\langle f\circ_{_V} g\in{\bf Hom}_{{\bf Set}_V}(X,Z):(f,g)\in{\bf Hom}_{{\bf Set}_V}(Y,Z)\times_V{\bf Hom}_{{\bf Set}_V}(X,Y)\rangle,$$ and finally we define $$\circ^V=\big\{\circ_{XYZ}\in\mathcal{P}\big({\bf Hom}_{{\bf Set}_V}(Y,Z)\times{\bf Hom}_{{\bf Set}_V}(X,Y)\times{\bf Hom}_{{\bf Set}_V}(X,Z)\big):(X,Y,Z)\in V\times V\times V\big\}.$$ \\
\end{enumerate}
\end{defn}

Note that ${\bf Set}_V$ can formally be considered as a tuple $$\big({\bf Ob}_{{\bf Set}_V},{\bf Hom}_{{\bf Set}_V},{\bf 1}^V,\circ^V\big)\in\mathcal{M}\times\mathcal{P}(\mathcal{P}(\bigcup\mathcal{M}))\times\mathcal{P}(\bigcup\mathcal{M})\times\mathcal{P}(\mathcal{P}(\bigcup\mathcal{M}\times\bigcup\mathcal{M}\times\bigcup\mathcal{M}))$$ since any function $f:X\to Y$ is just a member of $\mathcal{P}(X\times Y)$ satisfying certain conditions and, defining $${\bf Hom}_{{\bf Set}_V}(X,Y,Z)={\bf Hom}_{{\bf Set}_V}(Y,Z)\times{\bf Hom}_{{\bf Set}_V}(X,Y)$$ for brevity, we have that $${\bf Ob}_{{\bf Set}_V}=V\in\mathcal{M},$$ $${\bf Hom}_{{\bf Set}_V}\in\mathcal{P}(\mathcal{P}(V))\subseteq\mathcal{P}(\mathcal{P}(\bigcup\mathcal{M})),$$ $${\bf 1}^V\in\mathcal{P}(V)\subseteq\mathcal{P}(\bigcup\mathcal{M}),$$ $$\circ^V\in\mathcal{P}(\mathcal{P}({\bf Hom}_{{\bf Set}_V}(X,Y,Z)\times{\bf Hom}_{{\bf Set}_V}(X,Z)))\subseteq\mathcal{P}(\mathcal{P}(V\times V\times V))\subseteq\mathcal{P}(\mathcal{P}(\bigcup\mathcal{M}\times\bigcup\mathcal{M}\times\bigcup\mathcal{M})).$$ \\ Note the heavy usage of {\bf A10} and {\bf A11}; this formalization allows us to use class separation legitimately in the next definition. \\

\begin{defn}[{\it $2$-Category of Universes}]
We define a $2$-category $$\mathcal{V}\mathfrak{erse}$$ called the {\bf category of universes} as follows:
\begin{enumerate}
\item The objects of $\mathcal{V}\mathfrak{erse}$ are categories of sets in universes. That is, $${\bf Ob_{\mathcal{V}\mathfrak{erse}}}=\big\{{\bf Set}_V\in\mathcal{M}\times\mathcal{P}(\mathcal{P}(\bigcup\mathcal{M}))\times\mathcal{P}(\bigcup\mathcal{M})\times\mathcal{P}(\mathcal{P}(\bigcup\mathcal{M}\times\bigcup\mathcal{M}\times\bigcup\mathcal{M})):V\in\mathcal{M}\big\}.$$
\item For each pair of objects ${\bf Set}_V,{\bf Set}_W\in{\bf Ob}_{\mathcal{V}\mathfrak{erse}}$, we define the component category $$\mathcal{V}\mathfrak{erse}({\bf Set}_V,{\bf Set}_W)$$ to be the category of functors from ${\bf Set}_V$ to ${\bf Set}_W$. 
\item Identities are given by identity functors. 
\item Composition is given by composition of functors and the Godement product of natural transformations. \\
\end{enumerate}
\end{defn}

This $2$-category allows us to talk about the categories of sets in each universe; note that without {\bf A12} 'poorly structured' universes (i.e. those with very few axioms imposed on them besides {\bf A6}) wouldn't necessarily even support the standard notions of relation, function, etc., so {\bf A12} or another equivalent remedy is required for this part of the theory to go through cleanly. We could get around this by defining categories of sets only for universes supporting sufficient structure to allow their definition, but requiring that all universes actually model minimal set-theoretic axioms seems more canonical. {\bf A12*} would add internal powersets and consequently force ${\bf Set}_V$ to have power objects for all $V\in\mathcal{M}$, thusly yielding categories of sets that are topoi (note that topoi need not internally have replacement). \\

\section{The $3$-Category of $2$-Categories of $V$-Categories}

Finally, we define a tricategory whose objects are the $2$-categories of categories in each universe and whose component $2$-categories are bicategories of pseudofunctors, pseudonatural transformations and modifications between these $2$-categories of categories in each universe -- in what follows, we assume {\bf A12*}. Up first, we define the notion of a $V$-category. \\

\begin{defn}[{\it $V$-Category}]
Let $V$ be a universe. A {\bf $V$-category}, denoted $$\mathcal{C}_V,$$ consists of $V$-sets and $V$-functions such that they model the axioms of a category. That is, let $${\sf is\ a\ category}$$ denote the concatenation of the usual axioms for a category; a $V$-category then consists of the following data:
\begin{enumerate}
\item A $V$-class of objects, denoted ${\bf Ob}_{\mathcal{C}_V}.$
\item A $V$-class of arrows, denoted ${\bf Hom}_{\mathcal{C}_V}.$
\item Domain and codomain $V$-functions $dom_V,cod_V:{\bf Hom}_{\mathcal{C}(V)}\rightrightarrows{\bf Ob}_{\mathcal{C}_V}$.
\item An identity selecting $V$-function ${\bf 1_V}:{\bf Ob}_{\mathcal{C}_V}\to{\bf Hom}_{\mathcal{C}_V}$.
\item A composition $V$-function $\circ_V:{\bf Com}_{\mathcal{C}_V}\to{\bf Hom}_{\mathcal{C}_V}$.
\end{enumerate}
This data is required to satisfy $$({\bf Ob}_V,{\bf Hom}_V,dom_V,cod_V,{\bf 1}_V,\circ_V)\models{\sf is\ a\ category}.$$ Further, for all $X,Y\in{\bf Ob}_{\mathcal{C}(V)}$ we define $${\bf Hom}_{\mathcal{C}_V}(X,Y)=dom_V^{-1}(X)\cap cod_V^{-1}(Y).$$ We say that a $V$-category $\mathcal{C}_V$ is {\bf $V$-small} iff ${\bf Hom}_{\mathcal{C}_V}\in V$, {\bf $V$-locally small} iff ${\bf Hom}_{\mathcal{C}_V}(X,Y)\in V$ for all $X,Y\in{\bf Ob}_{\mathcal{C}_V}$, and {\bf $V$-large} otherwise.  \\
\end{defn}

So a $V$-category is just a category internal to $V$ -- we define $V$-functors and $V$-natural transformations similarly, relativizing all data to $V$ and replacing all instances of $\in$ with $\in_V$ -- note that (by internal extensionality) internal membership gives rise to a notion of 'internal equality' $=_V$ which may differ from external equality $=$. \\

\begin{defn}[{\it $V$-Functor Category}]
Let $V$ be a universe, with $\mathcal{C}_V$ and $\mathcal{D}_V$ $V$-categories. We define a category $$\mathcal{D}_V^{\mathcal{C}_V}$$ called the {\bf $V$-functor category from $\mathcal{C}_V$ to $\mathcal{D}_V$} as follows:
\begin{enumerate}
\item The objects of $\mathcal{D}_V^{\mathcal{C}_V}$ are $V$-functors $F_V:\mathcal{C}_V\to\mathcal{D}_V$. That is, $${\bf Ob}_{\mathcal{D}_V^{\mathcal{C}_V}}=\{F_V\in\mathcal{P}(V\times V)^2:F_V\ \text{is a $V$-functor from $\mathcal{C}_V$ to $\mathcal{D}_V$}\}.$$
\item The arrows $\alpha_V:F_V\Rightarrow G_V$ are given by $V$-natural transformations between $V$-functors. That is, $${\bf Hom}_{\mathcal{D}_V^{\mathcal{C}_V}}(F_V,G_V)=\{\alpha_V\in\mathcal{P}(V):\alpha_V\ \text{is a $V$-natural transformation from $F_V$ to $G_V$}\}.$$ 
\item Identities are given by internal identity natural transformations.
\item Composition is given by pointwise internal composition of internal natural transformations. \\
\end{enumerate}
\end{defn}

In this setting we can collect together all $V$-categories into one $2$-category without worrying about size, by virtue of the fact that we aren't gathering them into a '$V$-$2$-category' -- we're gathering them into an {\it external} $2$-category built out of {\it internal} $1$-categorical data. Note that a $V$-category $\mathcal{C}_V$ can formally be considered as a tuple $$({\bf Ob}_{\mathcal{C}_V},{\bf Hom}_{\mathcal{C}_V},dom_V,cod_V,{\bf 1}_V,\circ_V)\in\mathcal{P}(V)^2\times\mathcal{P}(V\times V)^4$$ since $${\bf Ob}_{\mathcal{C}_V},{\bf Hom}_{\mathcal{C}_V}\subseteq V\implies{\bf Ob}_{\mathcal{C}_V},{\bf Hom}_{\mathcal{C}_V}\in\mathcal{P}(V),$$ $$dom_V,cod_V\in\mathcal{P}({\bf Hom}_{\mathcal{C}_V}\times{\bf Ob}_{\mathcal{C}_V})\subseteq\mathcal{P}(V\times V),$$ $${\bf 1}_V\in\mathcal{P}({\bf Ob}_{\mathcal{C}_V}\times{\bf Hom}_{\mathcal{C}(V)})\subseteq\mathcal{P}(V\times V),$$ $$\circ_V\in\mathcal{P}({\bf Com}_{\mathcal{C}_V}\times{\bf Hom}_{\mathcal{C}_V})\subseteq\mathcal{P}(V\times V).$$ Further, by defining $${\bf 1}_{\mathcal{D}_V^{\mathcal{C}_V}}=\{1_{F_V}\in\mathcal{P}(V):F_V:\mathcal{C}_V\to\mathcal{D}_V\ \text{is a $V$-functor}\},$$ $$\circ_{\mathcal{D}_V^{\mathcal{C}_V}}=\{\circ_{XYZ}\in\mathcal{P}(V^3):(X,Y,Z)\in V^3\},$$ $V$-functor categories can be viewed as tuples $$\big({\bf Ob}_{\mathcal{D}_V^{\mathcal{C}_V}},{\bf Hom}_{\mathcal{D}_V^{\mathcal{C}_V}},{\bf 1}_{\mathcal{D}_V^{\mathcal{C}_V}},\circ_{\mathcal{D}_V^{\mathcal{C}_V}}\big)\in\mathcal{P}(\mathcal{P}(V\times V)^2)\times\mathcal{P}(\mathcal{P}(V))^2\times\mathcal{P}(\mathcal{P}(V^3)).$$  This allows us to use class separation legitimately in the next definition. \\

\begin{defn}[{\it $2$-Category of $V$-Categories}]
Let $V$ be a universe. We define a $2$-category $$\mathfrak{Cat}_V$$ called the {\bf $2$-category of $V$-categories} as follows:
\begin{enumerate}
\item The objects of $\mathfrak{Cat}_V$ are $V$-categories. That is, $${\bf Ob}_{\mathfrak{Cat}_V}=\{\mathcal{C}_V\in\mathcal{P}(V)^2\times\mathcal{P}(V\times V)^4:\mathcal{C}_V\ \text{is a $V$-category}\}.$$
\item The component categories $\mathfrak{Cat}_V(\mathcal{C}_V,\mathcal{D}_V)$ are given by $V$-functor categories $\mathcal{D}_V^{\mathcal{C}_V}$. That is, $${\bf Hom}_{\mathfrak{Cat}_V}=\{\mathcal{D}_V^{\mathcal{C}_V}\in\mathcal{P}(\mathcal{P}(V\times V)^2)\times\mathcal{P}(\mathcal{P}(V))^2\times\mathcal{P}(\mathcal{P}(V^3)):\mathcal{C}_V\ \text{and}\ \mathcal{D}_V\ \text{are $V$-categories}\}.$$
\item Identities are given by internal identity $V$-functors and internal identity $V$-natural transformations.
\item Composition is given by internal composition of $V$-functors and the internal Godement product of $V$-natural transformations. \\
\end{enumerate}
\end{defn}

Armed with $2$-categories of all (possibly large) $V$-categories for each universe $V$, we are prepared to tricategorify them. Note that by viewing the identities in a component category $\mathcal{D}_V^{\mathcal{C}_V}$ as identity selecting functors ${\sf 1}_{\mathcal{C}_V,\mathcal{D}_V}:\mathbb{1}_V\to\mathcal{D}_V^{\mathcal{C}_V}$ out of the free internal terminal category on a singleton in $V$ and recalling that composition is a functor $\Gamma_{\mathcal{C}_V,\mathcal{D}_V,\mathcal{E}_V}:\mathcal{E}_V^{\mathcal{D}_V}\times\mathcal{D}_V^{\mathcal{C}_V}\to\mathcal{E}_V^{\mathcal{C}_V}$, we can define $${\bf 1}_{\mathfrak{Cat}_V}=\{{\sf 1}_{\mathcal{C}_V,\mathcal{D}_V}\in\mathcal{P}(V\times V)^2:\mathcal{C}_V\ \text{and}\ \mathcal{D}_V\ \text{are $V$-categories}\},$$ $$\Gamma_{\mathfrak{Cat}_V}=\{\Gamma_{\mathcal{C}_V,\mathcal{D}_V,\mathcal{E}_V}\in\mathcal{P}(V^3)\times\mathcal{P}(V^3):\mathcal{C}_V,\mathcal{D}_V,\ \text{and}\ \mathcal{E}_V\ \text{are $V$-categories}\},$$ and recalling that $V\subseteq\bigcup\mathcal{M}$ for all universes $V$ allows us to view $\mathfrak{Cat}_V$ for an arbitrary universe $V$ as a tuple $$({\bf Ob}_{\mathfrak{Cat}_V},{\bf Hom}_{\mathfrak{Cat}_V},{\bf 1}_{\mathfrak{Cat}_V},\circ_{\mathfrak{Cat}_V})$$ in $$\mathfrak{X}=\mathcal{P}\Big(\mathcal{P}(\bigcup\mathcal{M})^2\times\mathcal{P}(\bigcup\mathcal{M}\times \bigcup\mathcal{M})^4\Big)\times\mathcal{P}\Big(\mathcal{P}(\mathcal{P}(\bigcup\mathcal{M}\times \bigcup\mathcal{M})^2)\times\mathcal{P}(\mathcal{P}(\bigcup\mathcal{M}))^2\times\mathcal{P}(\mathcal{P}(\bigcup\mathcal{M}^3))\Big)$$ $$\times\mathcal{P}\Big(\mathcal{P}(\bigcup\mathcal{M}\times \bigcup\mathcal{M})^2\Big)\times\mathcal{P}\Big(\mathcal{P}(\bigcup\mathcal{M}^3)\times\mathcal{P}(\bigcup\mathcal{M}^3)\Big),$$ once again allowing us to use class separation legitimately in the next definition. \\

\begin{defn}[{\it Tricategory of $2$-Categories of $V$-Categories}]
We define a tricategory $$\mathbb{Cat}$$ called the {\bf multiversal tricategory} as follows:
\begin{enumerate}
\item The objects of $\mathbb{Cat}$ are $2$-categories of $V$-categories for arbitrary universes $V$. That is, $${\bf Ob}_\mathbb{Cat}=\{\mathfrak{Cat}_V\in\mathfrak{X}:V\in\mathcal{M}\}.$$
\item The component bicategories are the bicategories of pseudofunctors, pseudonatural transformations and modifications between these bicategories of $V$-categories.
\item Identities are given by identity pseudofunctors, pseudonatural transformations and modifications.
\item Composition is given by composition of pseudofunctors, $2$-Godement products of pseudonatural transformations and horizontal composition of modifications. \\
\end{enumerate}
\end{defn}

We could obviously consider variations at differing levels of strictness for the component categories; choosing the 'pseudo' level is in no way necessary, and the stricter versions may be better behaved. \\

\section{The $n$-Category of $n-1$-Categories of$\dots$of $V$-Categories}

It is obviously possible to 'iterate' the above construction, considering the $4$-category of tricategories of bicategories of categories in each universe together with appropriately defined component $3$-functor categories, so on and so forth, at varying levels of 'internalization' -- we leave these topics for another paper. \\

\chapter{References}

\begin{itemize}
\item $\text{[Ham11]}$ - J. D. Hamkins. "The set-theoretic multiverse". Review of Symbolic Logic 5:416-449 (2012). arXiv:1108.4223 [math.LO]. \href{https://doi.org/10.1017/S1755020311000359}{https://doi.org/10.1017/S1755020311000359}
\item $\text{[GH11]}$ - V. Gitman, J. D. Hamkins. "A natural model of the multiverse axioms". arXiv:1104.4450 [math.LO]. \href{https://doi.org/10.48550/arXiv.1104.4450}{https://doi.org/10.48550/arXiv.1104.4450}
\item $\text{[Mul01]}$ - F. A. Muller. "Sets, Classes and Categories". British Journal for the Philosophy of Science 52 (2001) 539-573. \href{https://doi.org/10.1093/bjps/52.3.539}{https://doi.org/10.1093/bjps/52.3.539}
\item $\text{[LV61]}$ - A. Lévy, R. Vaught. "Principles of partial reflection in the set theories of Zermelo and Ackermann". Pacific J. Math. 11(3): 1045-1062 (1961). \href{https://doi.org/10.2140/pjm.1961.11.1045}{https://doi.org/10.2140/pjm.1961.11.1045}
\item $\text{[Ack56]}$ - W. Ackermann. "Zur Axiomatik der Mengenlehre". Math. Ann. 131, 336–345 (1956). \href{https://doi.org/10.1007/BF01350103}{https://doi.org/10.1007/BF01350103}
\item $\text{[Rein70]}$ - W. Reinhardt. "Ackermann's set theory equals ZF". Annals of Mathematical Logic, Volume 2, Issue 2, October 1970, Pages 189-249. \href{https://doi.org/10.1016/0003-4843(70)90011-2}{https://doi.org/10.1016/0003-4843(70)90011-2}
\item $\text{[Göd31]}$ - K. Gödel. "Über formal unentscheidbare Sätze der Principia Mathematica und verwandter Systeme, I", Monatshefte für Mathematik und Physik, v. 38 n. 1, pp. 173–198. \href{https://doi.org/10.1007/BF01700692}{https://doi.org/10.1007/BF01700692}
\end{itemize}

\end{document}